\documentclass[11pt, a4paper]{amsart}
\usepackage{fullpage}
\usepackage[colorlinks,linkcolor=blue,filecolor=blue,citecolor=blue,urlcolor=blue,bookmarks=false,hypertexnames]{hyperref}
\usepackage{tikz}
\usepackage{subcaption}
\usepackage[T1]{fontenc}
\usepackage{Alegreya}
\usepackage[libertine]{newtxmath}
\usepackage{paralist}

\usetikzlibrary{decorations.pathreplacing}
\tikzset{
	mybrace/.style={decorate,decoration={brace,aspect=#1}}
}
\newcommand\NN{\mathbb{N}}
\newcommand\naturals{\mathbb{N}}
\newcommand\ints{\mathbb{Z}}
\newcommand\reals{\mathbb{R}}

\def\MP{{\mathcal P}}
\def\MQ{{\mathcal Q}}
\def\MR{{\mathcal R}}

\newcommand{\calP}{\mathcal{P}}
\newcommand{\calQ}{\mathcal{Q}}
\newcommand{\calR}{\mathcal{R}}
\newcommand{\minus}{\ensuremath{\smallsetminus}}
\DeclareMathOperator{\Ext}{Ext}
\DeclareMathOperator{\Space}{sp}
\newcommand{\fraka}{\mathfrak{a}}

\makeatletter
\@addtoreset{equation}{section}
\def\theequation{\thesection.\@arabic \c@equation}
\makeatother

\theoremstyle{plain}

\newtheorem{theorem}[equation]{Theorem}

\newtheorem{Lemma}[equation]{Lemma}
\newtheorem{proposition}[equation]{Proposition}

\theoremstyle{definition}

\newtheorem{definition}[equation]{Definition}

\newtheorem{discussion}[equation]{Discussion}
\newenvironment{discussionbox}[1][]{%
    \begin{discussion}[#1]\pushQED{\qed}}{\popQED \end{discussion}}

\newtheorem{example}[equation]{Example}
\newenvironment{examplebox}[1][]{%
    \begin{example}[#1]\pushQED{\qed}}{\popQED \end{example}}

\newtheorem{question}[equation]{Question}

\newcommand{\define}[1]{\emph{#1}}

\title {The Charney-Davis conjecture for simple thin polyominoes}

\author{Manoj Kummini}
\address{Chennai Mathematical Institute, Siruseri, Tamilnadu 603103. India}
\email{mkummini@cmi.ac.in}

\author{Dharm Veer}
\address{Chennai Mathematical Institute, Siruseri, Tamilnadu 603103. India}
\email{dharm@cmi.ac.in}

\thanks{MK was partly supported by the grant CRG/2018/001592
	from Science and Engineering Research Board, India and
	by an Infosys Foundation fellowship.
	DV was partly supported by an Infosys Foundation fellowship.}


\begin{document}

\begin{abstract}
Let $\MP$ be a simple thin polyomino and $\Bbbk$ a field.
Let $R$ be the toric $\Bbbk$-algebra associated to $\MP$.
Write the Hilbert series of $R$ as
$h_{R}(t)/(1-t)^{\dim(R)}$.
We show that
$$(-1)^{\left\lfloor{\frac{\deg h_R(t)}{2}}\right\rfloor}h_{R}(-1) \geq 0$$
if $R$ is Gorenstein. This shows that the Gorenstein rings
associated to simple thin polyominoes satisfy the Charney-Davis conjecture.
\end{abstract}

	\maketitle

\section{Introduction}

The Charney-Davis conjecture~\cite[Conjecture~D]{ChDaeuler95}
asserts that if $h(t)$ is the $h$-polynomial of a flag simplicial homology
$(d-1)$-sphere, then $(-1)^{\lfloor \frac {d} {2}\rfloor} h(-1 ) \geq 0$.
Stanley~\cite[Problem~4]{Stanposprobs00} extended this conjecture to
Gorenstein${}^*$ flag simplicial complexes. Generalizing it further, Reiner
and Welker~\cite[Question~4.4]{ReWeCDNS05} posed the following:

\begin{question}
\label{question:CD}
Let $\Bbbk$ be a field and $R$ a standard graded Gorenstein Koszul
$\Bbbk$-algebra.
Write the Hilbert series of $R$ as $h_R(t)/(1-t)^{\dim(R)}$.
Is
$$
(-1)^{\left\lfloor{\frac{\deg h_R(t)}{2}}\right\rfloor}h_R(-1) \geq 0?
$$
\end{question}
We say that a standard graded Gorenstein Koszul $\Bbbk$-algebra $R$ is
\define{CD} if it gives an affirmative answer to the above question.
In this article, we show that the toric $\Bbbk$-algebras
associated to simple thin polyominoes are CD, when they are Gorenstein.

Suppose that, in the notation of Question~\ref{question:CD}, $\deg h_R(t)$
is odd.
Then $h_R(-1)=0$; see, e.g.,~\cite[Corollary~4.4.6]{BrHe:CM}.
Therefore Question~\ref{question:CD} is open only when $\deg h_R(t)$ is even.
See the bibliography of~\cite{ReWeCDNS05} and of~\cite{Stanposprobs00} for
various classes of rings that are CD.
A class of CD rings related to the ones we study in this paper are
Gorenstein Hibi
rings~\cite[Corollary~4.3]{BrandenSignGradedPosetsCDConj2004}.
Recently, D'Al\`i and
Venturello~\cite{DaliVenturelloKoszulGor2021} proved that the answer to
Question~\ref{question:CD} is negative in general.

A \define {cell} in $\reals^2$ is a set of the form $\{(x, y) \in \reals^2
\mid a \leq x \leq a+1, b\leq y \leq b+1\}$ where $(a,b ) \in \ints^2$.
Let $\calP$ be a finite collection of cells. Then $\calP$ determines a
unique topological subspace $\Space(\calP) := \cup_{C \in \calP } C$
of $\reals^2$. By abuse of terminology, we assign the topological
attributes to $\calP$ that $\Space(\calP)$ has.
We say that $\calP$ is a \define{polyomino}
if $\calP$ is connected and does not have a finite
cut-set~\cite[4.7.18]{StanEC1} (i.e., $\Space(\calP)$ has these properties).
There is a $\Bbbk$-algebra $\Bbbk[\calP]$ of finite-type associated to
$\calP$~\cite{QureshiPolyominoes2012}.

We say that a polyomino $\calP$ is \define{simple}
if $\Space(\calP)$ is simply connected;
it is \define{thin} if it does not have a $2 \times 2$ square
such as the one shown in Figure~\ref{figure:twoRooks}.
The S-property of simple thin polyominoes was introduced
in~\cite{RinaldoRomeoHilbSeriesThinPolyominoes2021}
to characterise such polyominoes $\calP$ for which $\Bbbk[\calP]$ is
Gorenstein.
If $\MP$ is simple then $\Bbbk[\MP]$ is a Koszul
algebra~\cite[Corollary~2.3]{QureshiShibutaShikamaPrime2017}.
Therefore it is natural to ask whether $\Bbbk[\calP]$ is CD if $\calP$ is a
simple thin polyomino with the S-property. In this regard, we show the
following:

\begin{theorem}
\label{theorem:cd}
Let $\calP$ be a collection of cells such that its connected components
are simple thin polyominoes with the S-property.
Then $\Bbbk[\calP]$ is CD.
\end{theorem}

Note that if $\calP_1, \ldots, \calP_m$ are the connected components of
$\calP$, then $\Bbbk[\calP] \simeq
\Bbbk[\calP_1] \otimes_\Bbbk \cdots \otimes_\Bbbk \Bbbk[\calP_m]$.
Therefore $\Bbbk[\calP]$ Gorenstein (respectively, Koszul) if and only if
$\Bbbk[\calP_i]$ is Gorenstein (respectively, Koszul) for each $i$.

Section~\ref{sec:preliminaries} contains the definitions and preliminaries.
Proof of the theorem is given in Section~\ref{sec:proofofthetheorem}.

\section{Preliminaries}\label{sec:preliminaries}

\subsection*{Gorenstein rings, Koszul rings, Hilbert series, etc.}
Let $R$ be a  finitely generated $\Bbbk$-algebra.
Throughout this paper, we deal with only standard
graded $\Bbbk$-algebras $R$, i.e., $R$ is generated as a $\Bbbk$-algebra
by homogeneous elements of degree $1$.
We say that $R$ is \define{Gorenstein} if it
is Cohen-Macaulay and $\Ext^{\dim R}_R(\Bbbk,R ) \simeq \Bbbk$.
We say that $R$ is \define{Koszul} if $\Bbbk$ has a linear free resolution
as an $R$-module.

The \define{Hilbert series} $H_R(t )$ of $R$ is the formal power series
$\sum_{i \in \naturals } \dim_\Bbbk R_i t^i$ where for each $i$, $R_i$ is
the finite-dimensional $\Bbbk$-vector-space of the homogeneous elements of
$R$ of degree $i$. There exists a unique polynomial $h_R(t)$ such that
\[
H_R(t ) = \frac{h_R(t)}{(1-t )^{\dim R } }.
\]
Question~\ref{question:CD} asks whether
$(-1)^{\left\lfloor{\frac{\deg h_R(t)}{2}}\right\rfloor}h_R(-1)$ is
non-negative for all Gorenstein Koszul algebras $R$.
Parenthetically, note that, if $R$ is Cohen-Macaulay, $\deg h_R(t)$ is the
Castelnuovo-Mumford regularity of $R$.

\subsection*{Polyominoes}

Let $\calP$ be a finite collection of cells. As mentioned in the
introduction, we treat $\calP$ interchangeably with the topological space
$\Space(\calP)$.
Let $S = \Bbbk[\{x_{i,j } \mid (i,j ) \in \calP \cap \ints^2\} ]$
be the standard graded polynomial ring in the variables
$x_{i,j }$.
Let $\fraka_\calP$ be the binomial ideal generated by the binomials
$x_{i,j} x_{k,l } - x_{k,j } x_{i,l }$ for all $(i,j ), (k,l) \in
\calP \cap \ints^2$ such that the rectangle with vertices
$(i,j)$, $(k,l)$, $(k,j )$ and $(i,l)$ is a subset of $\Space(\calP)$.
Define $\Bbbk[\calP] = S/\fraka_\calP$~\cite{QureshiPolyominoes2012}.
If $\calP$ is a simple polyomino, then $\Bbbk[\calP ]$ is a Koszul
Cohen-Macaulay integral
domain~\cite[Corollary~2.3]{QureshiShibutaShikamaPrime2017}.

For $k \in \naturals$, a \define{$k$-rook configuration} in $\calP$ is an
arrangement of $k$ rooks in pairwise non-attacking positions.
The \define{rook polynomial}
$r_\calP(t)$ of $\calP$ is $\sum_{k \in \naturals } r_k t^k$
where $r_k$ is the number of $k$-rook configurations in $\calP$.
The \define{rook number} $r(\calP)$ of $\calP$ is the degree of
$r_{\calP}(t)$, i.e., the largest $k$ such that there is a $k$-rook
configuration in $\calP$.
If $\MP$ is a simple thin polyomino, then $h_{\Bbbk[\calP]}(t) =
r_{\MP}(t)$~\cite[Theorem~1.1]{RinaldoRomeoHilbSeriesThinPolyominoes2021}.
\begin{figure}
\begin{center}
\begin{tikzpicture}[scale=2]
\draw[] (0,0)--(0,1)--(1,1)--(1,0)--(0,0) (0,.5)--(1,.5) (.5,0)--(.5,1);
\filldraw[black] (0,0) circle (.5pt) node[anchor=north]  {};
\filldraw[black] (.5,0) circle (.5pt) node[anchor=north] {};
\filldraw[black] (0,.5) circle (.5pt) node[anchor=east] {};
\filldraw[black] (.5,.5) circle (.5pt) node[anchor=west] {};
\filldraw[black] (1,1) circle (.5pt) node[anchor=north]  {};
\filldraw[black] (.5,1) circle (.5pt) node[anchor=north] {};
\filldraw[black] (1,.5) circle (.5pt) node[anchor=east] {};
\filldraw[black] (0,1) circle (.5pt) node[anchor=west] {};
\filldraw[black] (1,0) circle (.5pt) node[anchor=east] {};
\end{tikzpicture}
\caption{A $2\times 2$ square polyomino}\label{figure:twoRooks}
\end{center}
\end{figure}
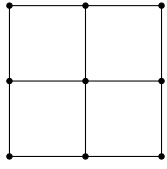

Let $C,D\in \MP$. We say that $C$ is a \define{neighbour} of $D$ if
$C\cap D$ is a line segment.
A \define{path} from $C$ to $D$ is a sequence of cells
$C = C_0, C_1, \ldots, C_m = D$ such that for all $i \neq j$, $C_i \neq
C_j$
and for all $1 \leq i \leq m$, $C_i$ is a neighbour of $C_{i-1}$.
If $\calP$ is a simple thin polyomino, then for all cells $C, D$ of
$\calP$, there is a unique path from $C$ to $D$.

A \define{inner interval} of $\calP$ is a subcollection $I$ of $\calP$ such
that $\Space(I)$ (which is a subspace of $\Space(\calP)$)
is a rectangle with vertices
$(i_1,j_1)$, $(i_1+1,j_1 )$, $(i_1,j_2 )$ and $(i_1+1, j_2 )$
or a rectangle with vertices
$(i_1,j_1)$, $(i_1,j_1+1)$, $(i_2,j_1 )$ and $(i_2, j_1+1)$
for some $i_1, i_2, j_1, j_2 \in \ints$ with
$i_1 < i_2$ and $j_1 < j_2$.
An inner interval of $\calP$ is \define{maximal}
if it is maximal under inclusion.

Let $\MP$ be a simple thin polyomino.
Observe that any cell of $\MP$ belongs to at most two maximal inner intervals.
A cell $C$ is said to be an {\em end-cell} of a maximal inner interval $I$
if $C\in I$ and $C$ has exactly one neighbour cell in $I$.
A cell of $\MP$ is called {\em single} if it belongs to exactly one maximal
inner interval of $\MP$.
We say that $\MP$ has the {\em S-property} if every maximal inner interval
of $\MP$ has exactly one single cell.
For a simple  thin polyomino $\MP$,  $\Bbbk[\MP]$ is Gorenstein if and only if $\MP$
has the
S-property~\cite[Theorem~4.2]{RinaldoRomeoHilbSeriesThinPolyominoes2021}.

\section{Proof of the  Theorem}\label{sec:proofofthetheorem}

We begin with an observation about how Hilbert series and rook polynomials
behave in disjoint unions of polyominoes.

\begin{proposition}
\label{proposition:disjUnion}
Let $\calP$ be a finite collection of cells.
Write $\calP_1, \ldots, \calP_m$ for the connected components.
Then:
\[
h_{\Bbbk[\calP]}(t) = \prod_{i=1}^m h_{\Bbbk[\calP_i]}(t)
\;\text{and}\;
r_{\calP}(t) = \prod_{i=1}^m r_{\calP_i}(t).
\]
In particular, if $\calP_i$ is a simple thin polyomino for each $i$, then
$h_{\Bbbk[\calP]}(t) = r_{\calP}(t)$.
\end{proposition}

\begin{proof}
Vertices of the $\calP_i$ are disjoint, so
$\Bbbk[\calP] \simeq
\Bbbk[\calP_1] \otimes_\Bbbk \cdots \otimes_\Bbbk \Bbbk[\calP_m]$.
Hence
$H_{\Bbbk[\calP]}(t) =
\prod_{i=1}^m H_{\Bbbk[\calP_i]}(t)$,
from which it follows that
$h_{\Bbbk[\calP]}(t) = \prod_{i=1}^m h_{\Bbbk[\calP_i]}(t)$.
Let $k \in \naturals$. Then $k$-rook configurations in $\calP$
corresponds to independent choices of $k_i$-rook configurations in
$\calP_i$ for each $1 \leq i \leq m$ and for each
tuple $(k_1, \ldots, k_m) \in \naturals^m$ with $\sum_i k_i = k$.
Hence
$r_{\calP}(t) = \prod_{i=1}^m r_{\calP_i}(t)$.
The final assertion now follows
from noting that for each $i$,
$h_{\Bbbk[\calP_i]}(t) = r_{\calP_i}(t)$
since $\calP_i$ is a simple thin
polyomino~\cite[Theorem~1.1]{RinaldoRomeoHilbSeriesThinPolyominoes2021}.
\end{proof}

Let $\calP$ be a simple thin polyomino.
In~\cite[Definition~3.4]{RinaldoRomeoHilbSeriesThinPolyominoes2021},
Rinaldo and Romeo introduced a notion of collapsing $\calP$ in a maximal
inner interval, and showed that
if $\calP$ has at least two maximal inner intervals,
then there exists a maximal inner interval in which $\calP$ is
collapsible~\cite[Proposition~3.7]{RinaldoRomeoHilbSeriesThinPolyominoes2021}.
We need a refinement  of this result for simple thin polyominoes with the
$S$-property, for which we
rephrase~\cite[Definition~3.4]{RinaldoRomeoHilbSeriesThinPolyominoes2021}
in a slightly different way.

\begin{definition}
\label{def:coll}
Let $\MP$ be a simple thin polyomino. A \define{collapse datum} on $\calP$
is a tuple $(I, J,\MP^I)$, where $I$ and $J$ are maximal inner
intervals and $\MP^I$ is a sub-polyomino of $\calP$
satisfying the following conditions:
\begin{enumerate}

  \item
$J$ is the only maximal inner interval of $\calP$ such that $I \cap J$ is a cell;

\item
$\calP^I \subseteq J$ and
$I \cap J \not \subset \calP^I$.

\item
$\calP \minus ( I \cup \calP^I)$ is a non-empty sub-polyomino of $\calP$.

\end{enumerate}
\end{definition}

Figure~\ref{fig:collapsible} gives an example of a collapse datum. Note
that since 
$\MP^I$ is a sub-polyomino of $\calP$
and $\calP^I \subseteq J$, it is an inner interval if it is non-empty. When $\calP$ has at least two maximal inner intervals, the maximal inner
intervals $I$ and $J$ defined in the Definition~\ref{def:coll} exist
by~\cite[Lemma~3.6]{RinaldoRomeoHilbSeriesThinPolyominoes2021}.

\begin{figure}%
		\centering
		\begin{tikzpicture}[scale=2]
		\draw[] (0,1.5)--(-.5,1.5)--(-.5,1)--(0,1)   (0,0)--(0,1.5)--(.5,1.5)

		(1,1.5)--(1.5,1.5)--(1.5,2)--(1,2)--(1,1.5)

		(2,.5)--(0,.5)

		(0,0)--(.5,0) (0,1)--(2,1)
		(0.5,0)--(0.5,1.5) (1,0.5)--(1,1)

		(1.5,.5)--(1.5,1) (1.5,0)--(1.5,-.5)--(1,.-.5)--(1,0)--(1.5,0)
		(2,0.5)--(2,1)

		(2.5,.5)--(2.5,1)--(3,1)--(3,.5)--(2.5,.5);

		\draw[dotted]   (2,1)--(2.5,1)(2.5,.5)--(2,0.5)
		 (1.5,1)--(1.5,1.5) (1,1)--(1,1.5)
		 (1,.5)--(1,0) (1.5,0)--(1.5,0.5);

		 \draw[mybrace=0.5, thick] (1.55,1.95) -- (3,1.95);

		\filldraw[black] (0,0) circle (.5pt) node[anchor=north]  {};
		\filldraw[black] (.5,0) circle (.5pt) node[anchor=north] {};
		\filldraw[black] (0,.5) circle (.5pt) node[anchor=east] {};
		\filldraw[black] (.5,.5) circle (.5pt) node[anchor=west] {};
		\filldraw[black] (1,1) circle (.5pt) node[anchor=north]  {};
		\filldraw[black] (.5,1) circle (.5pt) node[anchor=north] {};
		\filldraw[black] (1,.5) circle (.5pt) node[anchor=north] {};
		\filldraw[black] (0,1) circle (.5pt) node[anchor=east] {};

		\filldraw[black] (1.5,0.5) circle (.5pt) node[anchor=north] {};
		\filldraw[black] (0.5,1.5) circle (.5pt) node[anchor=east] {};
		\filldraw[black] (1.5,1.5) circle (.5pt) node[anchor=west] {};
		\filldraw[black] (1,1.5) circle (.5pt) node[anchor=north]  {};
		\filldraw[black] (1.5,0) circle (.5pt) node[anchor=north] {};
		\filldraw[black] (-.5,1.5) circle (.5pt) node[anchor=east] {};
		\filldraw[black] (0,1.5) circle (.5pt) node[anchor=east] {};
		\filldraw[black] (2,1) circle (.5pt) node[anchor=east] {};
		\filldraw[black] (2,.5) circle (.5pt) node[anchor=east] {};
		\filldraw[black] (-.5,1) circle (.5pt) node[anchor=west] {};
		\filldraw[black] (1.5,1) circle (.5pt) node[anchor=east] {};
		\filldraw[black] (1,0) circle (.5pt) node[anchor=east] {};

		\filldraw[] (1.25,1.3) circle (0pt) node[anchor=center]  {$\vdots$};
		\filldraw[] (1.25,.3) circle (0pt) node[anchor=center]  {$\vdots$};
		\filldraw[] (2.25,0.75) circle (0pt) node[anchor=center]  {$\ldots$};

		\filldraw[] (1.25,0.75) circle (0pt) node[anchor=center]  {$D$};
		\filldraw[] (1.25,2.25) circle (0pt) node[anchor=center]  {$I$};
		\filldraw[] (3.25,0.75) circle (0pt) node[anchor=center]  {$J$};
		\filldraw[] (2.22,2.25) circle (0pt) node[anchor=center]  {$\MP^I$};

		\end{tikzpicture}
	\caption{Collapse datum (\textit{cf}. Definition~\protect{\ref{def:coll}})}
    \label{fig:collapsible}
\end{figure}
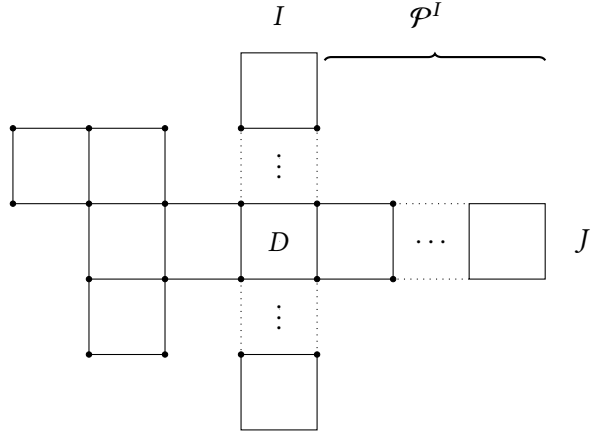

\begin{discussionbox}\label{rem:neighbourhoodofJ}
Let $\MP$ be a simple thin polyomino with S-property. Assume that $\MP$ is
not a cell.
Then it has a collapse datum $(I, J,\MP^I)$.
Since, additionally, $\MP$ has the S-property,
$I$ has exactly two cells and $\MP^I$ is either empty or a cell.
Write $I=\{C,D\}$ with $C$ denoting the single cell of $I$ and $\{D\} = I
\cap J$.
Denote the single cell of $J$ by $E$.
Let $C_1,\ldots,C_k$ be the cells of $J$ different from $D$ and $E$.
For $1\leq i \leq k$, let $B_i$ be the cell in
$\MP$ such that $B_i\notin J$ and $C_i$ is a neighbour cell of $B_i$; such
a $B_i$ must exist, since $C_i$ is not a single cell.
We now consider the various cases.

\begin{asparaenum}

\item
$\MP^I \neq \varnothing$.
Equivalently, $\MP^I = \{E\}$.
Then we may assume that $C_k$ is an end-cell of $J$ and for all $i \in \{1, \ldots, k-1 \}$,
$C_{i}$ and $C_{i+1}$ are neighbours. So
$B_{i}$ and $B_{i+1}$ can not be neighbours, since $\calP$ is thin.
Hence for all $i \in \{1, \ldots, k\}$ whether $B_i$ is above $C_i$ or is
below $C_i$ determined by whether $i$ is even or odd.
Therefore in the neighbourhood of $J$, $\calP$ is as
shown in Figure~\ref{fig:P_2nonempty}.

\item
$\MP^I = \varnothing$,
$C_k$ is an end-cell of $J$ and $E$ is a neighbour cell of $C_k$.
We may assume that for all $i \in \{1, \ldots, k-2 \}$,
$C_{i}$ and $C_{i+1}$ are neighbours, so
$B_{i}$ and $B_{i+1}$ cannot be neighbours.
Therefore, using the same considerations as in the above case, we see that
$\calP$ is as shown in Figure~\ref{fig:P_2emptycase1} in the neighbourhood
of $J$.

\item
$\MP^I = \varnothing$ and $E$ is an end-cell of $J$. Then, in the neighbourhood of $J$, $\calP$ is as shown in Figure~\ref{fig:P_2emptycase3}.

\item
$\MP^I = \varnothing$,
$C_k$ is an end-cell of $J$ and $E$ is not a
neighbour cell of $C_k$.
Then, in the neighbourhood of $J$, $\calP$ is
one of the figures in Figure~\ref{fig:P_2emptycase2}.

\end{asparaenum}
\end{discussionbox}

\begin{discussionbox}
By the \define{first end-cell} of $J$, we mean
\[
\begin{cases}
E, & \text{if}\; \calP^I = \{E \};\\
I \cap J, & \text{if}\; \calP^I = \varnothing.
\end{cases}
\]
(Note that in both of the above cases, the cell in question is an end-cell
of $J$.)
We call the other end-cell of $J$ the \define{second end-cell} of $J$.
If $E$ is the second end-cell of $J$, then $E$ has exactly
one neighbour cell.
If $C_k$ is the second end-cell of $J$ and $E$ is not a
neighbour cell of $C_k$, then $C_k$ has exactly two neighbour cells. If
$C_k$ is the second end-cell of $J$ and $E$ is a neighbour cell of $C_k$,
then $C_k$ has two or three neighbour cells; see
Figures~\ref{fig:P_2nonempty},~\ref{fig:P_2emptycase1}, \ref{fig:P_2emptycase3}
and~\ref{fig:P_2emptycase2}.
\end{discussionbox}

\begin{figure}%
		\centering
		\begin{tikzpicture}[scale=2]
		\draw[]  (.5,0)--(1.5,0) (2,0)--(3.5,0)--(3.5,.5)--(2,.5) (1.5,.5)--(.5,.5)--(.5,0)
		(3.5,0)--(3.5,.5) (2.5,0)--(2.5,.5) (3,0)--(3,.5) (2,0)--(2,.5) (1.5,0)--(1.5,.5)
		(1,0)--(1,.5)
		(3,.5)--(3,1)--(3.5,1)--(3.5,.5)
		(3,0)--(3,-.5)--(2.5,-.5)--(2.5,0)
		(2,.5)--(2,1)--(2.5,1)--(2.5,.5)
		(1,.5)--(1,1)--(1.5,1)--(1.5,.5)
		(1,0)--(1,-.5)--(.5,-.5)--(.5,0)
        (3.5,0)--(4,0)--(4,.5)--(3.5,.5);

		\draw[dotted] (1.5,0)--(2,0) (1.5,0.5)--(2,0.5) ;

		\filldraw[] (3.25,1.25) circle (0pt) node[anchor=center]  {$I$};
		\filldraw[] (3.75,0.25) circle (0pt) node[anchor=center]  {$E$};
		\filldraw[] (4.25,0.25) circle (0pt) node[anchor=center]  {$J$};

		\filldraw[] (3.25,0.25) circle (0pt) node[anchor=center]  {$D$};

		\filldraw[] (2.75,0.25) circle (0pt) node[anchor=center]  {$C_1$};
		\filldraw[] (2.25,0.25) circle (0pt) node[anchor=center]  {$C_2$};
		\filldraw[] (1.75,0.25) circle (0pt) node[anchor=center]  {$\cdots$};
		\filldraw[] (1.25,0.25) circle (0pt) node[anchor=center]  {$C_{k-1}$};
		\filldraw[] (.75,0.25) circle (0pt) node[anchor=center]  {$C_k$};
     	\filldraw[] (.75,-0.25) circle (0pt) node[anchor=center]  {$B_k$};

		\filldraw[] (1.25,0.75) circle (0pt) node[anchor=center]  {$B_{k-1}$};
		\filldraw[] (3.25,0.75) circle (0pt) node[anchor=center]  {$C$};
		\filldraw[] (2.25,0.75) circle (0pt) node[anchor=center]  {$B_2$};
		\filldraw[] (2.75,-0.25) circle (0pt) node[anchor=center]  {$B_1$};

		\end{tikzpicture}
		\caption{$\calP^I$ non-empty}\label{fig:P_2nonempty}
\end{figure}
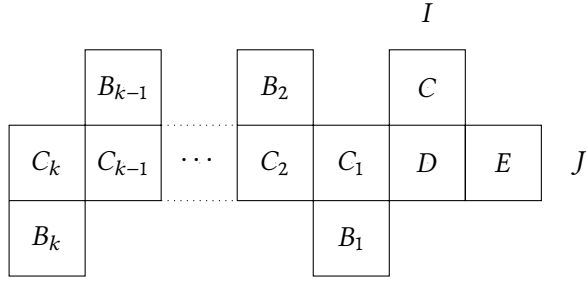

\begin{figure}%
		\centering
		\begin{tikzpicture}[scale=2]
		\draw[]  (.5,0)--(1.5,0) (2,0)--(3.5,0)--(3.5,.5)--(2,.5) (1.5,.5)--(.5,.5)--(.5,0)--(0,0)--(0,.5)--(.5,.5)
		(3.5,0)--(3.5,.5) (2.5,0)--(2.5,.5) (3,0)--(3,.5) (2,0)--(2,.5) (1.5,0)--(1.5,.5)
		(1,0)--(1,.5)
		(3,.5)--(3,1)--(3.5,1)--(3.5,.5)
		(3,0)--(3,-.5)--(2.5,-.5)--(2.5,0)
		(2,.5)--(2,1)--(2.5,1)--(2.5,.5)
		(1,.5)--(1,1)--(1.5,1)--(1.5,.5);

		\draw[dotted] (1.5,0)--(2,0) (1.5,0.5)--(2,0.5) (0,0)--(0,-.5)--(.5,-.5)--(.5,0)
		(0,.5)--(0,1)--(.5,1)--(.5,.5);

		\filldraw[] (3.25,1.25) circle (0pt) node[anchor=center]  {$I$};
		\filldraw[] (3.75,0.25) circle (0pt) node[anchor=center]  {$J$};

		\filldraw[] (3.25,0.25) circle (0pt) node[anchor=center]  {$D$};

		\filldraw[] (2.75,0.25) circle (0pt) node[anchor=center]  {$C_1$};
		\filldraw[] (2.25,0.25) circle (0pt) node[anchor=center]  {$C_2$};
		\filldraw[] (1.75,0.25) circle (0pt) node[anchor=center]  {$\cdots$};
		\filldraw[] (1.25,0.25) circle (0pt) node[anchor=center]  {$C_{k-1}$};
		\filldraw[] (.75,0.25) circle (0pt) node[anchor=center]  {$E$};
		\filldraw[] (.25,0.25) circle (0pt) node[anchor=center]  {$C_{k}$};
		\filldraw[] (.25,0.75) circle (0pt) node[anchor=center]  {$B_{k}$};
		\filldraw[] (.25,-0.25) circle (0pt) node[anchor=center]  {$B_{k+1}$};

		\filldraw[] (1.25,0.75) circle (0pt) node[anchor=center]  {$B_{k-1}$};
		\filldraw[] (3.25,0.75) circle (0pt) node[anchor=center]  {$C$};
		\filldraw[] (2.25,0.75) circle (0pt) node[anchor=center]  {$B_2$};
		\filldraw[] (2.75,-0.25) circle (0pt) node[anchor=center]  {$B_1$};

		\end{tikzpicture}
		\caption{$\calP^I$ is empty, $E$ is a neighbour of $C_k$ but not
        an end-cell of $J$}
		\label{fig:P_2emptycase1}
\end{figure}
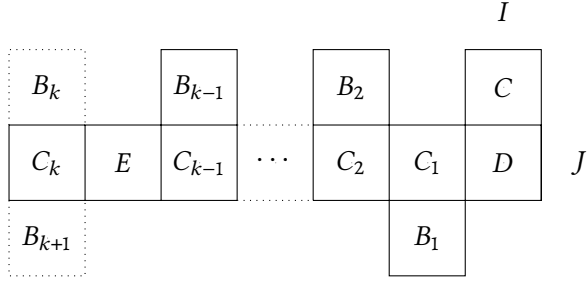

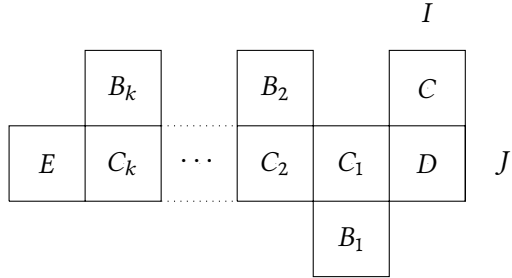
\begin{figure}
		\centering
		\begin{tikzpicture}[scale=2]
		\draw[]  (.5,0)--(1.5,0) (2,0)--(3.5,0)--(3.5,.5)--(2,.5) (1.5,.5)--(.5,.5)--(.5,0)
		(3.5,0)--(3.5,.5) (2.5,0)--(2.5,.5) (3,0)--(3,.5) (2,0)--(2,.5) (1.5,0)--(1.5,.5)
		(1,0)--(1,.5)
		(3,.5)--(3,1)--(3.5,1)--(3.5,.5)
		(3,0)--(3,-.5)--(2.5,-.5)--(2.5,0)
		(2,.5)--(2,1)--(2.5,1)--(2.5,.5)
		(1,.5)--(1,1)--(1.5,1)--(1.5,.5);

		\draw[dotted] (1.5,0)--(2,0) (1.5,0.5)--(2,0.5);

		\filldraw[] (3.25,1.25) circle (0pt) node[anchor=center]  {$I$};
		\filldraw[] (3.75,0.25) circle (0pt) node[anchor=center]  {$J$};

		\filldraw[] (3.25,0.25) circle (0pt) node[anchor=center]  {$D$};

		\filldraw[] (2.75,0.25) circle (0pt) node[anchor=center]  {$C_1$};
		\filldraw[] (2.25,0.25) circle (0pt) node[anchor=center]  {$C_2$};
		\filldraw[] (1.75,0.25) circle (0pt) node[anchor=center]  {$\cdots$};
		\filldraw[] (1.25,0.25) circle (0pt) node[anchor=center]  {$C_{k}$};
		\filldraw[] (.75,0.25) circle (0pt) node[anchor=center]  {$E$};

		\filldraw[] (1.25,0.75) circle (0pt) node[anchor=center]  {$B_{k}$};
		\filldraw[] (3.25,0.75) circle (0pt) node[anchor=center]  {$C$};
		\filldraw[] (2.25,0.75) circle (0pt) node[anchor=center]  {$B_2$};
		\filldraw[] (2.75,-0.25) circle (0pt) node[anchor=center]  {$B_1$};

		\end{tikzpicture}
		\caption{$\calP^I$ is empty, $E$ is an end-cell of $J$}
		\label{fig:P_2emptycase3}
\end{figure}

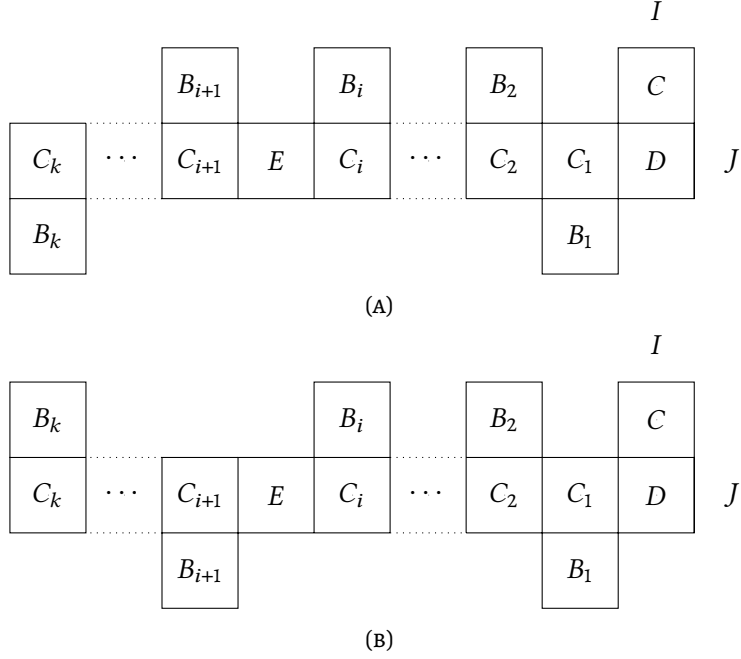
\begin{figure}%
	\begin{subfigure}[t]{10cm}
		\centering
		\begin{tikzpicture}[scale=2]
		\draw[]  (.5,0)--(1.5,0) (2,0)--(3.5,0)--(3.5,.5)--(2,.5) (1.5,.5)--(.5,.5)--(.5,0)--(0,0)--(0,.5)--(.5,.5) (-.5,0.5)--(-.5,0)--(-1,0)--(-1,.5)--(-.5,.5)
		(3.5,0)--(3.5,.5) (2.5,0)--(2.5,.5) (3,0)--(3,.5) (2,0)--(2,.5) (1.5,0)--(1.5,.5)
		(1,0)--(1,.5)
		(3,.5)--(3,1)--(3.5,1)--(3.5,.5)
		(3,0)--(3,-.5)--(2.5,-.5)--(2.5,0)
		(2,.5)--(2,1)--(2.5,1)--(2.5,.5)
		(1,.5)--(1,1)--(1.5,1)--(1.5,.5)
		(0,.5)--(0,1)--(.5,1)--(.5,.5)
		(-.5,0)--(-.5,-.5)--(-1,-.5)--(-1,0);

		\draw[dotted] (1.5,0)--(2,0) (1.5,0.5)--(2,0.5) (0,0)--(-.5,0) (0,.5)--(-.5,0.5);

		\filldraw[] (3.25,1.25) circle (0pt) node[anchor=center]  {$I$};
		\filldraw[] (3.75,0.25) circle (0pt) node[anchor=center]  {$J$};

		\filldraw[] (3.25,0.25) circle (0pt) node[anchor=center]  {$D$};

		\filldraw[] (2.75,0.25) circle (0pt) node[anchor=center]  {$C_1$};
		\filldraw[] (2.25,0.25) circle (0pt) node[anchor=center]  {$C_2$};
		\filldraw[] (1.75,0.25) circle (0pt) node[anchor=center]  {$\cdots$};
		\filldraw[] (1.25,0.25) circle (0pt) node[anchor=center]  {$C_{i}$};
		\filldraw[] (.75,0.25) circle (0pt) node[anchor=center]  {$E$};
		\filldraw[] (.25,0.25) circle (0pt) node[anchor=center]  {$C_{i+1}$};
		\filldraw[] (.25,0.75) circle (0pt) node[anchor=center]  {$B_{i+1}$};
		\filldraw[] (-.25,0.25) circle (0pt) node[anchor=center]  {$\cdots$};

		\filldraw[] (-.75,0.25) circle (0pt) node[anchor=center]  {$C_k$};
		\filldraw[] (-.75,-0.25) circle (0pt) node[anchor=center]  {$B_k$};
		\filldraw[] (1.25,0.75) circle (0pt) node[anchor=center]  {$B_{i}$};
		\filldraw[] (3.25,0.75) circle (0pt) node[anchor=center]  {$C$};
		\filldraw[] (2.25,0.75) circle (0pt) node[anchor=center]  {$B_2$};
		\filldraw[] (2.75,-0.25) circle (0pt) node[anchor=center]  {$B_1$};

		\end{tikzpicture}
		\caption{}%
	\end{subfigure}
	\\
	\begin{subfigure}[t]{10cm}
		\centering
		\begin{tikzpicture}[scale=2]
		\draw[]  (.5,0)--(1.5,0) (2,0)--(3.5,0)--(3.5,.5)--(2,.5) (1.5,.5)--(.5,.5)--(.5,0)--(0,0)--(0,.5)--(.5,.5) (-.5,0.5)--(-.5,0)--(-1,0)--(-1,.5)--(-.5,.5)
		(3.5,0)--(3.5,.5) (2.5,0)--(2.5,.5) (3,0)--(3,.5) (2,0)--(2,.5) (1.5,0)--(1.5,.5)
		(1,0)--(1,.5)
		(3,.5)--(3,1)--(3.5,1)--(3.5,.5)
		(3,0)--(3,-.5)--(2.5,-.5)--(2.5,0)
		(2,.5)--(2,1)--(2.5,1)--(2.5,.5)
		(1,.5)--(1,1)--(1.5,1)--(1.5,.5)
		(0,0)--(0,-.5)--(.5,-.5)--(.5,0)
		(-.5,.5)--(-.5,1)--(-1,1)--(-1,.5);

		\draw[dotted] (1.5,0)--(2,0) (1.5,0.5)--(2,0.5) (0,0)--(-.5,0) (0,.5)--(-.5,0.5);

		\filldraw[] (3.25,1.25) circle (0pt) node[anchor=center]  {$I$};
		\filldraw[] (3.75,0.25) circle (0pt) node[anchor=center]  {$J$};

		\filldraw[] (3.25,0.25) circle (0pt) node[anchor=center]  {$D$};

		\filldraw[] (2.75,0.25) circle (0pt) node[anchor=center]  {$C_1$};
		\filldraw[] (2.25,0.25) circle (0pt) node[anchor=center]  {$C_2$};
		\filldraw[] (1.75,0.25) circle (0pt) node[anchor=center]  {$\cdots$};
		\filldraw[] (1.25,0.25) circle (0pt) node[anchor=center]  {$C_{i}$};
		\filldraw[] (.75,0.25) circle (0pt) node[anchor=center]  {$E$};
		\filldraw[] (.25,0.25) circle (0pt) node[anchor=center]  {$C_{i+1}$};
		\filldraw[] (.25,-0.25) circle (0pt) node[anchor=center]  {$B_{i+1}$};

		\filldraw[] (-.25,0.25) circle (0pt) node[anchor=center]  {$\cdots$};
		\filldraw[] (-.75,0.25) circle (0pt) node[anchor=center]  {$C_k$};
		\filldraw[] (-.75,0.75) circle (0pt) node[anchor=center]  {$B_k$};
		\filldraw[] (1.25,0.75) circle (0pt) node[anchor=center]  {$B_{i}$};
		\filldraw[] (3.25,0.75) circle (0pt) node[anchor=center]  {$C$};
		\filldraw[] (2.25,0.75) circle (0pt) node[anchor=center]  {$B_2$};
		\filldraw[] (2.75,-0.25) circle (0pt) node[anchor=center]  {$B_1$};

		\end{tikzpicture}
		\caption{}%
	\end{subfigure}
		\caption{$\calP^I$ is empty, $E$ is not a neighbour of $C_k$}
	\label{fig:P_2emptycase2}
\end{figure}

The next lemma shows that simple thin polyominoes with the $S$-property
have a special collapse datum.
See Figure~\ref{figure:lemmaneeded} for an example to illustrate that
condition~\eqref{lem:collapsible:three} of the next lemma cannot be
deleted.

\begin{Lemma}\label{lem:collapsible}
Let $\MP$ be a simple thin polyomino with S-property. Assume that $\MP$ is
not a cell. Then there exists a collapse datum $(I, J,\MP^I)$ of
$\MP$ such that one of the following holds:

\begin{enumerate}
\item
\label{lem:collapsible:two}
The second end-cell of $J$ has at most two neighbour cells.

\item
\label{lem:collapsible:three}
If the second end-cell of $J$ has three neighbour cells, then one of its
neighbour cells is both a single cell and an end-cell of the maximal inner
interval containing it.

\end{enumerate}

\end{Lemma}

\begin{proof}
We proceed by induction on number of cells of $\MP$.
Let $(I, J,\MP^{I})$ be a collapse datum of $\MP$.
If the neighbourhood of $J$ in $\calP$ looks like the ones given in
Figures~\ref{fig:P_2nonempty}, \ref{fig:P_2emptycase3}
or~\ref{fig:P_2emptycase2}, 
then~\eqref{lem:collapsible:two} holds.
Therefore we may assume that we are in the situation of
Figure~\ref{fig:P_2emptycase1}.  Let $E$ be the single cell of $J$, $C_k$
be the second end-cell of $J$. We may assume that  $B_k$ and $B_{k+1}$ as
marked in Figure~\ref{fig:P_2emptycase1} exist.

If $B_k$ is both a single cell and an end-cell of the maximal inner interval
that contains $\{B_k, C_k, B_{k+1}\}$ 
then~\eqref{lem:collapsible:three} holds.
Hence we may assume that either $B_k$ is not a single cell or it is not an
end-cell of the maximal inner interval that contains 
$\{B_k, C_k, B_{k+1}\}$.
Similarly for $B_{k+1}$. 
Let
$$
\MP'= (\MP\setminus \{A\in \MP :
\text{the (unique) path between $A$ and $E$ does not contain $C_k$}\}) \cup
\{E\}.
$$

(E.g., in Figure~\ref{fig:P_2emptycase1}, $\MP'$ is the sub-polyomino
comprising $E$ and the cells reachable from $E$ through $C_k$.)
Observe that $\MP'$ is a simple thin polyomino. We first show that $\MP'$
has the S-property. Let $L$ be a maximal inner interval of $\MP'$.
Then $L$ is a maximal inner interval of $\MP$ or $L=\{C_k,E\}$.
In both cases,  $L$ has a unique single cell.
Thus, $\MP'$ has the S-property.
Also note that $\MP'$ is not a cell.
The number of cells in $\MP'$ are strictly less than the number of
cells in  $\MP$.
Therefore, by induction, there exists a collapse datum
$(I', J',\MP'^{I'})$ of $\MP'$
satisfying the assertion of the lemma for $\calP'$.

Note that $I' \neq \{C_k,E \}$ and $J' \neq \{C_k,E \}$; therefore $I'$ and
$J'$ are maximal inner intervals of $\calP$.
Let $J_1$ be a maximal inner interval of $\calP$ such that $J_1 \cap I'$ is
a cell.
Since $\calP$ is simple, $J_1 \subset \calP'$, so $J_1$ is a maximal inner
interval of $\calP'$. Hence $J_1 = J'$.
Note that $\calP'^{I'} \subseteq J'$ 
and that $I' \cap J' \not \subset \calP'^{I'}$.
Moreover, 
since $J' \neq \{C_k,E \}$, it follows that 
$\{C_k,E \} \subseteq \calP' \minus (I' \cup \calP'^{I'})$; hence
$\calP \minus (I' \cup \calP'^{I'})$
which equals
\[
(\calP' \minus (I' \cup \calP'^{I'})) \cup
\{A\in \MP :
\text{the path between $A$ and $E$ does not contain $C_k$}\}
\]
is a non-empty sub-polyomino of $\calP$.
Hence 
$(I', J',\MP'^{I'})$ is a collapse datum of $\MP$ that
satisfies the assertion of the lemma for $\calP$.
\end{proof}

\begin{figure}%
	\centering
	\begin{tikzpicture}[scale=2]
	\draw[]   (0,0)--(0,1.5)--(.5,1.5)

		(1,.5)--(1,0)--(1.5,0)--(1.5,.5)

	(1.5,.5)--(0,.5)

	(0,0)--(.5,0) (0,1)--(1.5,1)
	(0.5,0)--(0.5,1.5) (1,0.5)--(1,1)

	(1.5,.5)--(1.5,1)

	(0,1.5)--(-1,1.5)--(-1,0) (-.5,1.5)--(-.5,0)--(-1,0)
	(0,1)--(-1,1)
	(-.5,1)--(-2,1)--(-2,0)--(-1.5,0)--(-1.5,1)
	(-.5,.5)--(-2,.5)

	;

	\end{tikzpicture}
	\caption{}\label{figure:lemmaneeded}
\end{figure}

\begin{discussionbox}
\label{discussionbox:deletion}
Let $\MP$ be a simple thin polyomino and $C$ be a single cell in $\MP$. Let
$r_{\MP,C}(t)$ be the polynomial $\mathop{\sum}_{k\in \NN}r_kt^k$, where
$r_k$ is the number of $k$-rook configurations in $\MP$ that have a rook at
$C$. Let $r_{\MP,\widehat{C}}(t)$ be the polynomial $\mathop{\sum}_{k\in
\NN}r_kt^k$, where $r_k$ is the number of $k$-rook configurations in $\MP$
that have no rook at $C$. Then,
	$$r_{\MP}(t) = r_{\MP,\widehat{C}}(t) + r_{\MP,C}(t).$$

Let $I$ be the maximal inner interval of $\MP$ such that $C\in I$. Let
$r_{\MP,\widehat{I}}(t)$ be the polynomial $\mathop{\sum}_{k\in
\NN}r_kt^k$, where $r_k$ is the number of $k$-rook configurations in $\MP$
that has no rook at any cell of $I$. Note that $r_{\MP,C}(t) =
r_{\MP,\widehat{I}}(t)t$.
Hence,
\begin{equation}
\label{equation:deletion}
r_{\MP}(t) = r_{\MP,\widehat{C}}(t) + r_{\MP,\widehat{I}}(t)t.
\qedhere
\end{equation}

\end{discussionbox}

\begin{examplebox} 
\label{example:deletion}
We illustrate the above definitions now.
Let $\MP$ be the polyomino as shown in the Figure~\ref{figure:example}. 
Note that $C$ is a single cell in $\MP$. 
The polynomials $r_{\MP,\widehat{C}}(t)$ and $r_{\MP,C}(t)$ 
are calculated in Table~\ref{table:examplerook}.
The unique maximal inner interval $I$ containing $C$ is $\{B,C\}$.
Hence $r_{\MP,\widehat{I}}(t) = 1+t$, since this is the rook polynomial of
the polyomino consisting of just the cell $A$.
On the other hand, the number of $k$-rook configurations in $\MP$
for $k=0, 1, 2$ are, respectively, $1, 3, 1$; hence $r_{\MP}(t) = 1 + 3t +
t^2$. We thus see that
\[
r_{\MP}(t) = r_{\MP,\widehat{C}}(t) + r_{\MP,C}(t)
=r_{\MP,\widehat{C}}(t) + r_{\MP,\widehat{I}}(t)t.
\qedhere
\]
\end{examplebox}

 \begin{table}
 \begin{center}
     \begin{tabular}{|c | p{4cm} |c| p{4cm}| c|}
         \hline
         $k$ &        $k$-rook configurations that have a rook at $C$ &
         number     &    $k$-rook configurations that do not have a rook at $C$ & number \\ 
         \hline
         0 & There are no 0-rook configurations
that have a rook at $C$
         &  0 & $\varnothing$ & 1\\
         1 & $\{C\}$ & 1 & $\{A\}, \{B\}$ & 2\\
         2 & $\{C, A\}$ & 1 & none& 0\\
         $k\geq 3$ & none & 0 & none & 0\\ 
         \hline
& $r_{P,C}(t)$ & $t+t^2$ & $r_{P,\widehat{C}}(t)$ & $1+2t$ \\
         \hline
     \end{tabular}
  \end{center}
  \caption{Calculation of $r_{P,C }(t)$ and $r_{P,\widehat{C}}(t)$}
\label{table:examplerook}
  \end{table}

    \begin{figure}
   \centering
	\begin{tikzpicture}[scale=1]
        \draw[]   (0,0)--(0,2)--(2,2)--(2,1)--(1,1)--(1,0)--(0,0)
        (0,1)--(1,1)--(1,2)
	;
    \filldraw[] (.5,.5) circle (0pt) node[anchor=center] {$A$};
    \filldraw[] (.5,1.5) circle (0pt) node[anchor=center]{$B$};
     \filldraw[] (1.5,1.5) circle (0pt) node[anchor=center]{$C$};
   \end{tikzpicture}
	\caption{Polyomino from Example~\protect{\ref{example:deletion}}}
    \label{figure:example}
\end{figure}
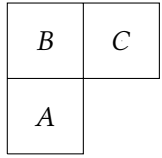
We now wish to express
$r_{\MP,\widehat{C}}(t)$ and $r_{\MP,\widehat{I}}(t)$ as the rook polynomials
of polyominoes when $\calP$ has the S-property.

\begin{discussionbox}
\label{dis:single}
Let $\MP$ be a simple thin polyomino that has the S-property. Let $(I,
J,\MP^{I})$ be a collapse datum of $\MP$ satisfying the
conclusion of Lemma~\ref{lem:collapsible}.
Let $C$ and $D$ be the cells of $I$, with $C$
being the single cell.
Let $E$ be the single cell of $J$.
Let $r_{\MP,\widehat{C}}(t)$ and $r_{\MP,\widehat{I}}(t)$ be as defined in
Discussion~\ref{discussionbox:deletion}.

Write $\calQ = \MP\setminus\{C\}$.
Then $r_{\MP,\widehat{C}}(t) = r_{\calQ}(t)$.
If $\MP^{I}$ is empty, define $\MR$ to be the polyomino $\calP \minus I$.
Otherwise, i.e. if
$\MP^{I}$ is a cell $E$, define $\calR$ to be the polyomino
$\calP \minus \{ C , E\}$.
Then $r_{\MP,\widehat{I}}(t) =r_{\MR}(t)$.
Thus~\eqref{equation:deletion} becomes
\begin{equation}
\label{equation:rPQR}
r_{\MP}(t) = r_{\calQ}(t) + r_{\calR}(t)t.
\end{equation}

Note that $\calQ$ does not have the S-property, so we cannot use an
inductive argument to prove Theorem~\ref{theorem:cd} directly.
Hence we need to rewrite
$r_{\calQ }(t )$ in terms of smaller polyominoes.
To this end, we observe that
$D$ is a single cell of the maximal inner interval $J$ inside $\calQ$.
Therefore, by~\eqref{equation:deletion},
\begin{equation}
\label{equation:deletionR}
r_{\calQ }(t ) =
r_{\MQ,\widehat{D}}(t) + r_{\MQ,\widehat{J}}(t) t .
\end{equation}
We note that
\begin{equation}
\label{equation:rRD}
r_{\MQ,\widehat{D}}(t) =r_{\MR}(t).
\qedhere
\end{equation}
\end{discussionbox}

We now find an expression for $r_{\MQ,\widehat{J}}(t)$.
Let $E, C_1,\ldots,C_k$ be the other cells of $J$ in $\MP$.
$E$ denotes the single cell of $J$ in $\MP$.
For $1\leq i \leq k-1$, let $B_i$ be the cell
in $\MP$ such that $B_i\notin J$ and $C_i$ is a neighbour cell of $B_i$.
(See Discussion~\ref{rem:neighbourhoodofJ} and
Figures~\ref{fig:P_2nonempty},
\ref{fig:P_2emptycase1},
\ref{fig:P_2emptycase3},
and~\ref{fig:P_2emptycase2} for notational conventions.)
When $E$ is the second end-cell or $C_k$ is an end-cell with two neighbour
cells, let $B_k$ be the cell in $\MP$ such that $B_k\notin J$ and $C_k$ is
a neighbour cell of $B_k$. When $C_k$ is an end-cell with three neighbour
cells, let $B_k$ and $B_{k+1}$ be the cells in $\MP$ such that $B_k,
B_{k+1}\notin J$ and $C_k$ is a neighbour cell of $B_k$ and $B_{k+1}$. In
the case when $C_k$ has three neighbour cells, by
Lemma~\ref{lem:collapsible}, we may assume that $B_{k+1}$ is both a single
cell and an end-cell of the maximal inner interval containing it.

Now for all $1\leq i \leq k-1$, define
$$\MQ_i := \{A\in \MQ : \text{the path between $A$ and $B_i$ does not contain $C_i$}\}.$$
Also, define
$$\widetilde{\MQ_k} := \{A\in \MQ : \text{the path between $A$ and $B_k$ does not contain $C_k$}\}.$$
When $E$ is the second end-cell or $C_k$ is an end-cell with two neighbour
cells, define $\MQ_k = \widetilde{\MQ_k}$. When $C_k$ is an end-cell with
three neighbour cells, let $\{a,b,a',b'\}$ be the vertices of $C_k$ where
$a,b \in V(B_k)$ and $a',b' \in V(B_{k+1})$. We define $\MQ_k$ as the
polyomino obtained from $\widetilde{\MQ_k} \cup \{B_{k+1}\}$ by the
identification of the vertices $a$ and $b$ of $V(B_k)$ with the vertices
$a'$ and $b'$ of $V(B_{k+1})$, respectively, by translating the
cell $B_{k+1}$.

\begin{Lemma}\label{lem:sproperty}
With notation as above, we have the following:
\begin{enumerate}

\item
\label{lem:sproperty:conn}
$\calQ_1, \ldots, \calQ_k$ are precisely the connected components of $\calQ
\minus J$.

\item
\label{lem:sproperty:sprop}
For each $1 \leq i \leq k$,
$\MQ_i$ is a simple thin polyomino with the S-property.

\item
\label{lem:sproperty:prod}
$r_{\MQ,\widehat{J}}(t) = \mathop\prod_{i=1}^{k} r_{\MQ_i}(t)$ and
$\sum_{i=1}^{k}r(\MQ_i)= r(\MP)-2$.
\end{enumerate}
\end{Lemma}

\begin{proof}

\eqref{lem:sproperty:conn}:
Each $\calQ_i$ is connected, $\calQ_i \cap \calQ_j = \varnothing$ for all
$i \neq j$ (since $P$ is simple)
and $\calQ \minus J = \calQ_1 \cup \cdots \cup \calQ_k$.

\eqref{lem:sproperty:sprop}:
Since $\MQ$ is simple thin, so is $\MQ_i$. Let $L$ be a maximal inner
interval of $\MQ_i$. Then, either $L$ is a maximal  inner interval of $\MP$
or $L = L'\setminus \{C_i\}$, where $L'$ is a maximal inner interval of
$\MP$. In both cases, $L$ has a unique single cell. Hence $\MQ_i$ has
the S-property.

\eqref{lem:sproperty:prod}:
By~\eqref{lem:sproperty:conn}
$r_{\MQ,\widehat{J}}(t)$ is the rook polynomial of $\calQ \minus J$.
Thus, by
Proposition~\ref{proposition:disjUnion}, $r_{\MQ,\widehat{J}}(t) =
\mathop\prod_{i=1}^{k} r_{\MQ_i}(t)$.
For any $k$-rook configuration $\alpha$ of $\MQ \setminus J$,
we note that
$\alpha \cup \{C, E\}$ is a $(k+2)$-rook configuration of $\MP$.
Hence $\sum_{i=1}^{k}r(\MQ_i)\leq r(\MP)-2$. On the
other hand, let $\beta$ be a $r(\MP)$-rook configuration of $\MP$. Since
$\MP$ has the S-property, $\beta$ is the only $r(\MP)$-rook configuration of
$\MP$ and $\beta$ is the collection of all single  cells  of $\MP$. Then,
$\beta\setminus\{C,E\}$ is a rook configuration of $\MQ \setminus J$.
Therefore $\sum_{i=1}^{k}r(\MQ_i)\geq r(\MP)-2$.
\end{proof}

We are now ready to prove Theorem~\ref{theorem:cd}.
\begin{proof}[Proof of the Theorem]
By Proposition~\ref{proposition:disjUnion}, we may assume that $\MP$ is a
simple thin polyomino with the S-property. Let $h_{\Bbbk[\calP
]}/(1-t)^{\dim(\Bbbk[\MP])}$ be the Hilbert series of $\Bbbk[\MP]$. By
\cite[Theorem\ 1.1]{RinaldoRomeoHilbSeriesThinPolyominoes2021},
$h_{\Bbbk[\calP ]}(t)=r_{\MP}(t)$. We proceed by induction on the rook
number $r(\MP)$.
If $r(\MP)$ is odd (in particular if $r(\calP )=1$),
then $r_{\MP}(-1)=0$.
Hence we may assume that $r(\calP )$ is even.
Let $(I, J,\MP^{I})$ be a collapse datum of $\MP$ satisfying
the conclusion of Lemma~\ref{lem:collapsible}.
Apply Discussion~\ref{dis:single}, adopting its notation.
Let $\calQ$ and $\calR$ be as in Discussion~\ref{dis:single}.
Then, by~\eqref{equation:rPQR}
$$
r_{\MP}(t) = r_{\calQ}(t) + r_{\MR}(t)t.
$$
Now apply Discussion~\ref{discussionbox:deletion} to the single cell $D$ of
the maximal inner interval $J$ of $\calQ$. By~\eqref{equation:deletionR},
\eqref{equation:rRD} and Lemma~\ref{lem:sproperty}, we see that
$$r_{\MP}(t) = (1+t)r_{\MR}(t) + t\prod_{i=1}^{k}	r_{\MQ_i}(t).$$

By Lemma~\ref{lem:sproperty} and induction hypothesis, $\Bbbk[\MQ_i]$ is CD
for all $1\leq i\leq k$. If $r(\MQ_i)$ is odd for some $i$, then
$r_{\MP}(-1)=0$.
Therefore we may assume that $r(\MQ_i)$ is even for all $i$.
\begin{align*}
(-1)^{\lfloor\frac{r(\MP)}{2}\rfloor}r_{\MP}(-1) & =
(-1)^{\frac{r(\MP)}{2}+1}\prod_{i=1}^{k} r_{\MQ_i}(-1)
\\
& =(-1)^{\frac{r(\MP)-2}{2}}\prod_{i=1}^{k} r_{\MQ_i}(-1)
\\
& = \prod_{i=1}^{k}	(-1)^{\frac{r(\MQ_i)}{2}}r_{\MQ_i}(-1) & \text{by
Lemma~\protect{\ref{lem:sproperty}}}\\
& \geq 0 & \text{by induction}.
\end{align*}
This completes the proof of the theorem.
\end{proof}

\def\cfudot#1{\ifmmode\setbox7\hbox{$\accent"5E#1$}\else
  \setbox7\hbox{\accent"5E#1}\penalty 10000\relax\fi\raise 1\ht7
  \hbox{\raise.1ex\hbox to 1\wd7{\hss.\hss}}\penalty 10000 \hskip-1\wd7\penalty
  10000\box7}

\end{document}